\pdfoutput=1
\RequirePackage{ifpdf}
\ifpdf % We are running pdfTeX in pdf mode
\documentclass[pdftex]{sigma}
\else
\documentclass{sigma}
\fi

\usepackage{tikz}

\newtheorem{Theorem}{Theorem}[section]
\newtheorem{Corollary}[Theorem]{Corollary}
\newtheorem{Lemma}[Theorem]{Lemma}

 { \theoremstyle{definition}

 }

\begin{document}

\allowdisplaybreaks

\renewcommand{\thefootnote}{}

\renewcommand{\PaperNumber}{097}

\FirstPageHeading

\ShortArticleName{Fun Problems in Geometry and Beyond}

\ArticleName{Fun Problems in Geometry and Beyond\footnote{This paper is a~contribution to the Special Issue on Algebra, Topology, and Dynamics in Interaction in honor of Dmitry Fuchs. The full collection is available at \href{https://www.emis.de/journals/SIGMA/Fuchs.html}{https://www.emis.de/journals/SIGMA/Fuchs.html}}}

\Author{Boris KHESIN~$^\dag$ and Serge TABACHNIKOV~$^\ddag$}

\AuthorNameForHeading{B.~Khesin and S.~Tabachnikov}

\Address{$^\dag$~Department of Mathematics, University of Toronto, Toronto, ON M5S 2E4, Canada}
\EmailD{\href{mailto:khesin@math.toronto.edu}{khesin@math.toronto.edu}}
\URLaddressD{\url{http://www.math.toronto.edu/khesin/}}

\Address{$^\ddag$~Department of Mathematics, Pennsylvania State University, University Park, PA 16802, USA}
\EmailD{\href{mailto:tabachni@math.psu.edu}{tabachni@math.psu.edu}}
\URLaddressD{\url{http://www.personal.psu.edu/sot2/}}

\ArticleDates{Received November 17, 2019; Published online December 11, 2019}

\Abstract{We discuss fun problems, vaguely related to notions and theorems of a course in differential geometry. This paper can be regarded as a weekend ``treasure chest'' supplemen\-ting the course weekday lecture notes. The problems and solutions are not original, while their relation to the course might be so.}

\Keywords{clocks; spot it!; hunters; parking; frames; tangents; algebra; geometry}

\Classification{00A08; 00A09}

\begin{flushright}\begin{minipage}{66mm}
\it To Dmitry Borisovich Fuchs \\
on the occasion of his 80th anniversary
\end{minipage} \end{flushright}

\begin{flushright}\begin{minipage}{66mm}
\includegraphics{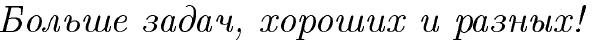} \\ %{\it \foreignlanguage{russian}{Больше задач, хороших и разных!}}\\
 attributed to Winston Churchill (1918)
\end{minipage}
\end{flushright}

\renewcommand{\thefootnote}{\arabic{footnote}}
\setcounter{footnote}{0}

This is a belated birthday present to Dmitry Borisovich Fuchs, a mathematician young at heart.\footnote{According to V.~Arnold, a mathematician is young as long as he reads works other than his own!}

Many of these problems belong to the mathematical folklore. We did not attempt to trace them back to their origins (which might be impossible anyway), and the given references are not necessarily the first mentioning of these problems.

\section{Problems}

\begin{enumerate}\itemsep=0pt
\item A wall clock has the minute and hour hands of equal length. Nevertheless almost always one can tell the time by the precise positions of the two hands. How many times a day one cannot tell the time?

\item How to prove the fundamental theorem of algebra by purely differential geometric methods: via the implicit function theorem and the notion of an immersed discriminant of codim~2.

\item The game ``Spot it!'' is a fast-paced party game. The concept is simple: each card in the deck has 8 symbols,
and there is exactly one matching symbol between any two cards in the deck, see Fig.~\ref{spot-it}.
The goal of every participant is to be the quickest to find the match between two cards.
Explain the game via geometry of the projective plane over the field of 7 elements, $\mathbb F_7 \mathbb P^2.$

\begin{figure}[hbtp]\centering
\includegraphics[height=1.9in]{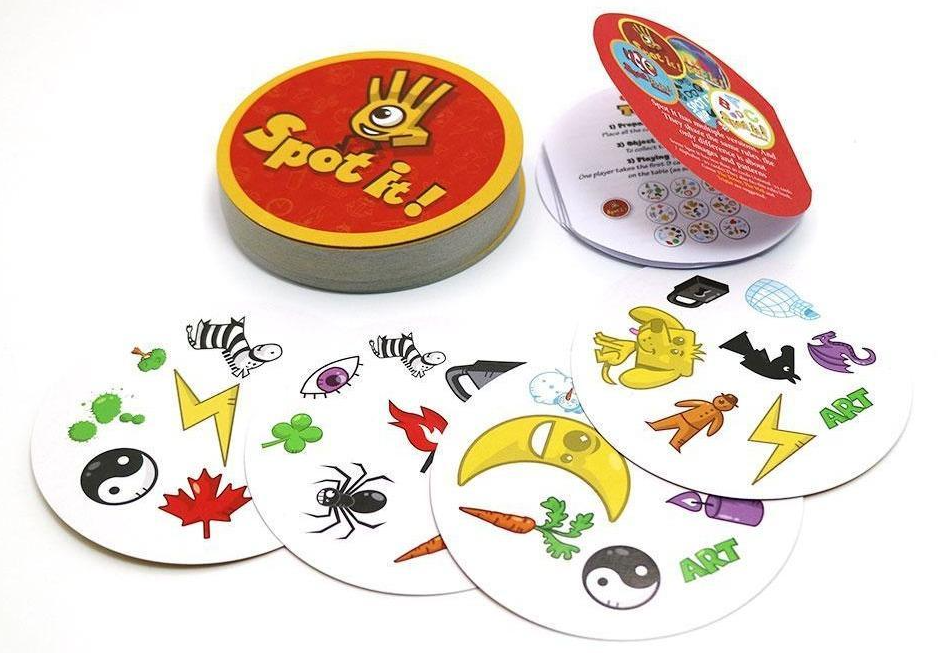}
\caption{Cards of the game ``Spot it!''}\label{spot-it}
\end{figure}

\item How to hang a picture frame on the wall on two (three, any number of) nails, so that if either of the nails is taken out, the picture frame falls?

\item Prove that the altitudes of a spherical triangle are concurrent as a consequence of the Jacobi identity in the Lie algebra of motions of the sphere.
\item Explain parallel parking with the help of the Frobenius theorem.

\item
The rules of the Moscow subway restrict the size $L\le 150$~cm of a luggage (rectangular parallelepiped) that can be taken along, where~$L$ is the sum of the parallelepiped's measurements: $L= \text{length}+\text{width}+\text{height}$ (see \url{https://www.mosmetro.ru/info/}, Section~2.10.1). May one cheat by putting a box of a larger size inside a smaller one?
\item Queuing up in the Moscow subway: how to separate 1/3 of a flow of people by simple dividers splitting any flux into two equal parts? How to be 1/3 Spanish?

\item A hunter started from his tent and went 10~km strictly south, then 10~km strictly west, then 10~km strictly north, and arrived at his original tent. Find the locus of possible locations of hunter's tent.

\item A family has two kids. It is known that one of them is a boy born on Friday. What is the probability that the other child is also a boy?
What if we knew that one of the kids is a~boy born between 6~pm and 7~pm on Friday~-- what would be the probability for the other kid to be a boy in that case?
\item Can one place a square table on an uneven surface? What about a rectangular one?
\item Given a cone (a mountain), throw a loop on it and pull it tight (to climb). Clearly, if the cone is sufficiently acute, it will work, and if it is considerably obtuse, the loop will slide away. What is the borderline angle?
\item Given two nested ovals, prove that there is a line tangent to the inner one so that the tangency point is halfway between its intersections with the outer oval (in fact, there are at least two such lines).
\item Wrap a closed non-stretchable string around a convex set, pull tight, and move the pulled end around to construct a curve. Prove that the string makes equal angles with this curve.
 \item Prove Ivory's lemma: in a quadrilateral, made of two confocal ellipses and two confocal hyperbolas, the two diagonals are equal.
\item Prove the Chasles--Reye theorem: pick two points on an ellipse and draw the tangent lines from them to a confocal ellipse. Then the other two vertices of the resulting quadrilateral lie on a confocal hyperbola, and the quadrilateral is circumscribed about a circle, see Fig.~\ref{Chasles}.

\begin{figure}[hbtp]\centering
\includegraphics{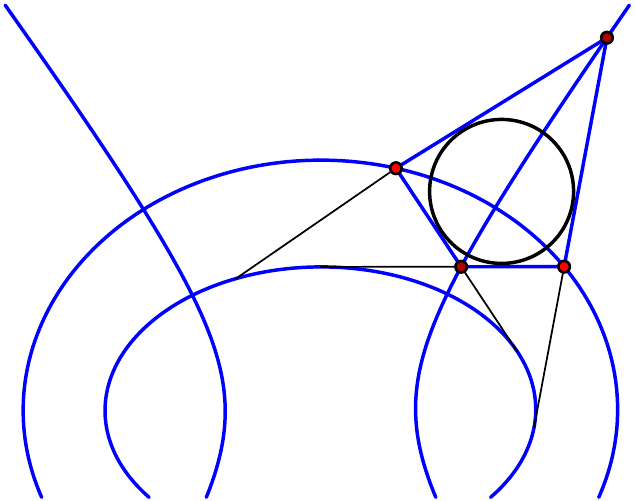}
\caption{The Chasles--Reye theorem.}\label{Chasles}
\end{figure}

\item An equitangent problem: is there a convex plane body $\mathcal B$ with the property that one can walk around it so that, at every moment the right tangent segment to $\mathcal B$ is longer than the left one?
\item The pentagram map sends a polygon to the polygon whose vertices are the intersection points of the combinatorially short (skip a vertex) diagonals of the initial polygon, see Fig.~\ref{hepta} and \cite{schwartz1992pentagram}. A polygon $P$ in the projective plane is called projectively self-dual if there exists a projective map from the projective plane to the dual plane that takes $P$ to the dual polygon $P^*$ (whose vertices are the sides of $P$ and whose sides are the vertices of $P$).
Show that pentagons in the projective plane are projectively self-dual, and the pentagram map takes a pentagon to a projectively equivalent one.

\begin{figure}[hbtp]\centering
\includegraphics{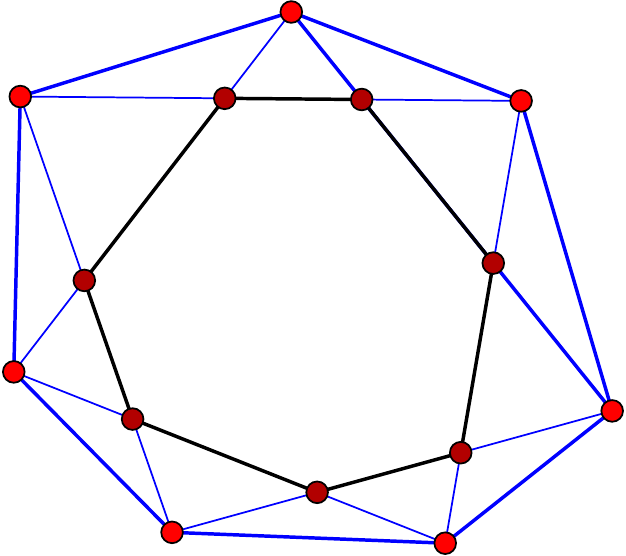}
\caption{The pentagram map.}\label{hepta}
\end{figure}

\item It is easy to make a M\"obius strip from a long and narrow paper rectangle, but it is impossible to make it from a square. What is the smallest ratio of length to width of a paper strip from which one can make a M\"obius strip? Paper is non-stretchable, that is, the M\"obius strip is a developable surface.
\end{enumerate}

\section{Solutions}

In the solutions below we occasionally leave some questions, in lieu of lengthy explanations, which the inquisitive reader can regard as small exercises on the way.

\subsection{Problem 1: Clocks with hands of equal length}

First note that during the first 12 hours the clock hands coincide exactly 11 times (why?).

The configuration space of clock hands is the product of two circles, i.e., a torus. If $x$ $({\rm mod}\, 2\pi)$ is the coordinate of the hour hand, and $y$ $({\rm mod}\, 2\pi)$ is the coordinate of the minute hand, the points $(x,y)$ on this torus, corresponding to possible positions of the hour and minute hands, are defined by the equation $y=12x$, since the minute hand is 12 times as fast as the hour one. The corresponding `graph' corresponds to the curve (cycle) on the torus, winding one time in the horizontal direction and 12 times in the vertical one.

Now, if the hour and minute hands are interchanged, the corresponding curve is given by the equation $x=12y$ and it winds 12 times in the horizontal direction and one time in the vertical one, see Fig.~\ref{clock-hands}.

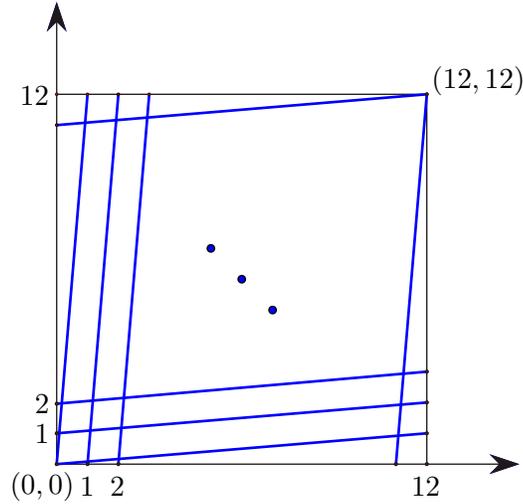
\begin{figure}[hbtp]\centering
 \begin{tikzpicture}
\node at (0,0){\includegraphics{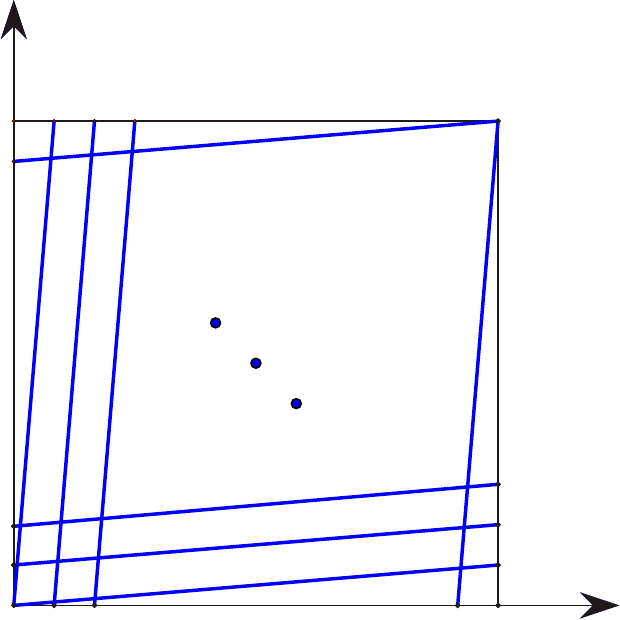}};
\node at (2.6,2.1){$(12,12)$};
\node at (-3.3,1.9){$12$};
\node at (-3.2,-2.2){$2$};
\node at (-3.2,-2.6){$1$};
\node at (-3.2,-3.3){$(0,0)$};
\node at (-2.6,-3.3){$1$};
\node at (-2.2,-3.3){$2$};
\node at (1.9,-3.3){$12$};
\end{tikzpicture}
\caption{Graphs of possible positions for different choices of minute and hour hands. Here $x,y\in [0,12]$ are measured in hours, rather than in radians $[0,2\pi]$.}\label{clock-hands}
\end{figure}

One cannot tell the time from clocks where these two curves intersect. There are exactly 143 intersections. Indeed, on the square $[0,2\pi]\times [0,2\pi]$ there are $12\times 12=144$ intersections of the graphs $y=12x$ for $x\in [0,2\pi]$, $y \bmod 2\pi$ and $x=12y$ for $y\in [0,2\pi]$, $x \bmod 2\pi$, but the points~$(0,0)$ and $(2\pi,2\pi)$ stand for the same position of hands on the clock.

However, this is not a complete answer yet. The points on the diagonal correspond to coinciding minute and hour hands, and at those moments one can tell the time. There are exactly~11 moments like this. Hence during the first 12 hours one cannot tell the time $143-11=132$ times. One has to double this answer for the whole day of 24 hours: one cannot tell the time $2\times 132=264$ times, i.e., roughly every~5.5 minutes.

\subsection{Problem 2: The fundamental theorem of algebra}

\begin{Theorem}
Every nonconstant polynomial over $\mathbb C$ has a root.
\end{Theorem}

The following proof is an adaptation of an argument from~\cite{AVG}.

\begin{proof}Consider the space of polynomials $P(z,a)=z^n+a_{n-1}z^{n-1}+\dots +a_0$ of $z\in {\mathbb C}$ of degree~$n$ with complex coefficients $ a_{n-1},\dots ,a_0$.

In the space ${\mathbb C}^{n+1}=\{(z,a_{n-1},\dots ,a_0)\}$ consider the set $\Sigma$ cut out by the 2 equations $P(z,a)=Q(z,a)=0$, where $Q(z,a) := P'(z,a)=\partial P/\partial z =nz^{n-1}+(n-1)a_{n-1}z^{n-2}+\dots +a_1$. This is the space of polynomials with coefficients $a$, that have double root at $z$.

\begin{Lemma} \label{lemma1}$\Sigma$ is a smooth submanifold in ${\mathbb C}^{n+1}$ of complex codimension $2$ $($i.e., of complex dimension $n-1)$.
\end{Lemma}

\begin{proof}[Proof of Lemma~\ref{lemma1}] It suffices to check that for the map
\[
(P,Q)\colon \ {\mathbb C}^{n+1}\to {\mathbb C}^2,\qquad (z,a)\mapsto (P(z,a),Q(z,a))
\]
$(0,0)$ is a regular value. Actually, for this map all values in ${\mathbb C}^2$ are regular, since $\nabla P$ and $\nabla Q$ are everywhere noncollinear in ${\mathbb C}^{n+1}$. Indeed, $\nabla P=(*,\dots ,*,1)$, while $\nabla Q=(*,\dots ,*,1,0)$, where the last two derivatives are with respect to $a_1$ and $a_0$.
\end{proof}

Project $\pi\colon \Sigma^{n-1}\to{\mathbb C}^n_a$ by the ``forgetful map'' $(z,a)\mapsto a$. The image $\Delta:=\pi(\Sigma)$ is the space of all polynomials which have a double root. This image $\Delta$ is a surface in~${\mathbb C}^n_a$, which has complex dimension (not greater than)~$n-1$, that of~$\Sigma$. (Actually, this hypersurface $\Delta^{n-1}$ is singular, as the forgetful map of~$\Sigma^{n-1}$ to~${\mathbb C}^n_a$ is not an embedding. For $n=3$ it is a cylinder over a cusp, while for $n=4$ it is a cylinder over the singular surface called the swallowtail, see Fig.~\ref{swallowtail}.)

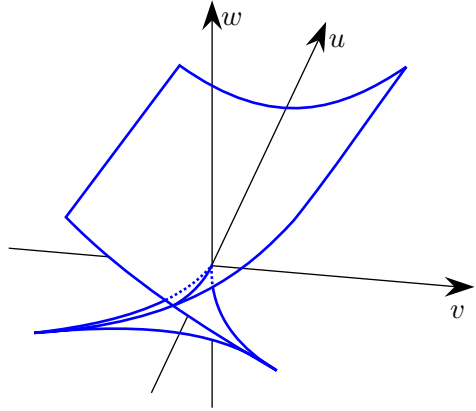
\begin{figure}[hbtp]\centering
\begin{tikzpicture}
\node at (0,0){\includegraphics{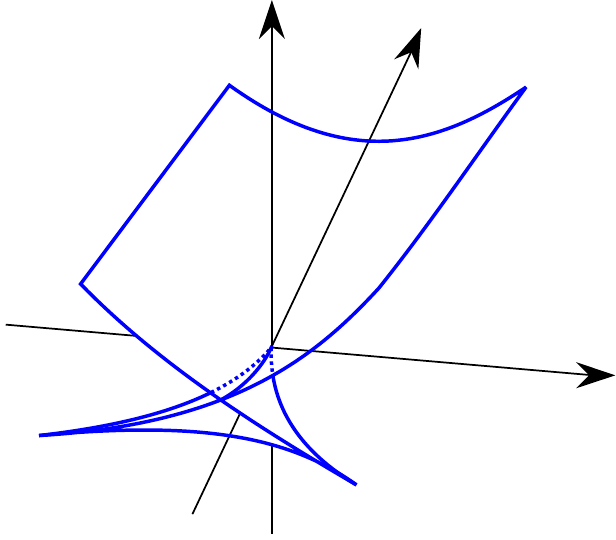}};
\node at (-0.1,2.5){$w$};
\node at (1.3,2.2){$u$};
\node at (2.9,-1.4){$v$};
\end{tikzpicture}
\caption{The swallowtail, discriminant $\Delta$, is the set of coefficients $(u,v,w)$ of polynomials $z^4+u z^2+v z+w$ with double roots.}\label{swallowtail}
\end{figure}

\begin{Corollary}\label{corollary2} The surface $\Delta^{n-1}\subset {\mathbb C}^n_a$ is of complex codimension~$1$, i.e., of real codimension~$2$ in ${\mathbb C}^n_a$, and therefore its complement in ${\mathbb C}^n_a$ is connected.
\end{Corollary}

\begin{Lemma}\label{lemma3}For $a\not\in\Delta$, the roots of the polynomial $P(z,a)$ depend smoothly on $a$.
\end{Lemma}

\begin{proof}[Proof of Lemma~\ref{lemma3}] This is the implicit function theorem: the equation $P(z,a)=0$ locally defines~$z$ as a function of $a$ provided
that $\partial P(z,a)/\partial z \not=0$. But the latter is exactly the condition that the corresponding $(z,a)\not\in\Sigma$, i.e., $a\not\in\Delta$.
\end{proof}

\begin{Corollary}\label{corollary4}Any two polynomials of degree $n$, lying outside of the complex hypersurface $\Delta^{n-1}\subset {\mathbb C}^n_a$, have the same number of roots.
\end{Corollary}

Indeed, connect them by a smooth path staying away from $\Delta^{n-1}$ (which is possible due to Corollary~\ref{corollary2}). On the way, the roots change smoothly, i.e., they cannot collide, appear, or disappear.

Note that the polynomial $P_0:=z^n-1$ has $n$ simple roots. All polynomials outside of $\Delta^{n-1}$ can be connected to it, hence they also have the same number of roots (and therefore, at least one).

Finally, it remains to prove the theorem for $P\in \Delta$. But this is evident by definition of $\Delta$: this surface consists of polynomials which have (at least one) double root.
\end{proof}

\subsection{Problem 3: Spot it!}
This game is based on finding the intersections of any two lines in the projective plane over the field of 7 elements: $\mathbb F_7 \mathbb P^2$.

In more detail, recall the decomposition for a projective space:
\[
\mathbb{RP}^n={\mathbb R}^n\cup {\mathbb R}^{n-1}\cup\cdots \cup{\mathbb R}\cup \{0\}.
\]
One can define a projective space for any field $\mathbb F$. It has a similar decomposition:
\[
\mathbb{FP}^n=\mathbb{F}^n\cup \mathbb{F}^{n-1}\cup\cdots \cup\mathbb{F}\cup \{0\}.
\]
In particular, we may take $\mathbb F$ to be a finite field.
Projective lines $\mathbb{FP}^1$ in $\mathbb{FP}^n$ are defined similarly to the case $\mathbb{F}={\mathbb R}$.

Here are two exercises:
\begin{enumerate}\itemsep=0pt
\item[$i)$] If $\mathbb F$ is a finite field with $q$ elements, give a formula for the cardinality of $ \mathbb{FP}^n$.
\item[$ii)$] For a finite field $\mathbb F$ with $q$ elements give a formula for the number of projective lines $\mathbb{FP}^1$ in the projective plane~$\mathbb{FP}^2$. (Hint: there is a duality between points and lines in the projective plane.)
\end{enumerate}

Now we are back to the party game ``Spot it!''. It contains a deck of 55 cards, each of which has 8 different symbols (pictures) printed on its front.
Think of each card as a projective line~$\mathbb{FP}^1$ in the projective plane~$\mathbb{FP}^2$ over a field $\mathbb{F}:=\mathbb{F}_7$ with 7 elements,
while the symbols on each card stand for points belonging to that line.
The property that any two cards have a unique symbol in common means that every two lines in the projective plane have a unique point of intersection. To find this intersection point is the objective of players, who need
to spot the common symbol for a pair of cards as quickly as possible.
Now show that one could actually enlarge the collection to 57 cards!

\subsection{Problem 4: A frame on the wall}

Denote the clockwise kink of the rope about the first nail $a$ by $A$, and the
 clockwise kink about the second nail $b$ by $B$, while $A^{-1}$ and $B^{-1}$ stand for the counterclockwise kinks of the rope about the corresponding nails.
Let the path of the rope $\gamma$ be the commutator $[A,B]:=ABA^{-1}B^{-1}$ (it is the group commutator in the fundamental group of the plane without two points, $\mathbb R^2\setminus\{a,b\}$), cf. Fig.~\ref{rope}.

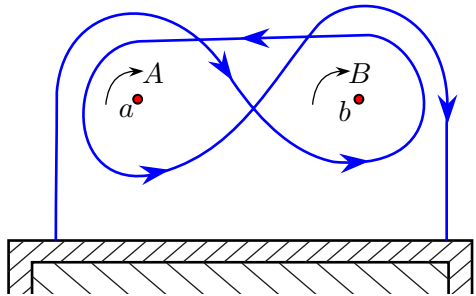
\begin{figure}[hbtp]\centering
\begin{tikzpicture}
\node at (0,0){\includegraphics{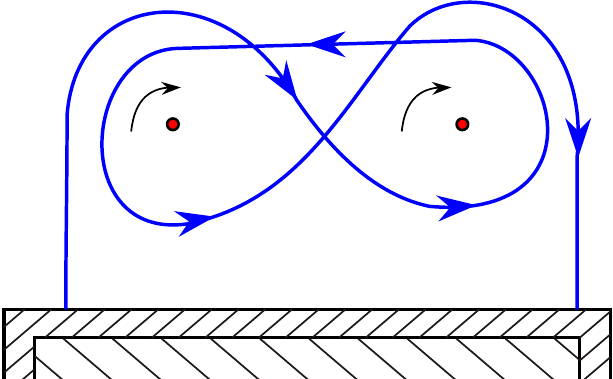}};
\node at (-1.5,0.5){$a$};
\node at (-1.15,1.0){$A$};
\node at (1.4,0.5){$b$};
\node at (1.6,1.0){$B$};
\end{tikzpicture}
\caption{A rope forms the commutator of kinks $\big[A,B^{-1}\big]=AB^{-1}A^{-1}B$ around the two nails $a$ and $b$ to hang a frame.}\label{rope}
\end{figure}

The absence of the first nail $a$ means that the corresponding kink is contractible, or identity, $A=I$. Then the commutator is the identity as well, $[A,B]=I$, and the frame falls (the path $\gamma=[A,B]$ of the rope in $\mathbb R^2\setminus \{b\}$, i.e., when the point $a$ is ``filled'', becomes contractible). Similarly, once there is no nail~$b$, i.e., $B=I$, again $[A,B]=I$. This solution is not unique: e.g., one can have the rope $\gamma$ to be the path $[A,B]^2$ or $[A,B^{-1}]$, which have the same property: $\gamma\not=I$ in general, but once $A=I$ or $B=I$, one has $\gamma=I$.

For three nails $a$, $b$, and $c$, we have three generators: $A$, $B$, and $C$. Then we set
\begin{gather*}\begin{split}&
\gamma:=[[A,B],C]= \big[ABA^{-1}B^{-1}, C\big]=ABA^{-1}B^{-1}C\big(ABA^{-1}B^{-1}\big)^{-1}C^{-1}\\
& \hphantom{\gamma}{}=ABA^{-1}B^{-1}CBAB^{-1}A^{-1}C^{-1}.\end{split}
\end{gather*}
Again, once either $A$, $B$ or $C$ is $I$, then $\gamma=I$. (If for two nails the path in Fig.~\ref{rope} might have been guessed, for three nails it is already highly unlikely without invoking the notion of the commutator.)

Similarly, one can accommodate any number of nails. As above, solutions are not unique. For instance, for~4 nails one can use $\gamma=[[[A,B],C],D]$ or $\gamma=[[A,B],[C,D]]$. They all have the property that once either of $A$, $B$, $C$, or $D$ is equal to~$I$ (i.e., one of the nails is removed), then $\gamma=I$ (and the frame falls). See \cite{DD} for a detailed discussion.

\subsection{Problem 5: Triangle altitudes and the Jacobi identity}

In the spherical geometry, one has the duality between (oriented) great circles (lines) and their poles, see Fig.~\ref{flag}.

\begin{figure}[hbtp]\centering
\begin{tikzpicture}
\node at (0,0){\includegraphics{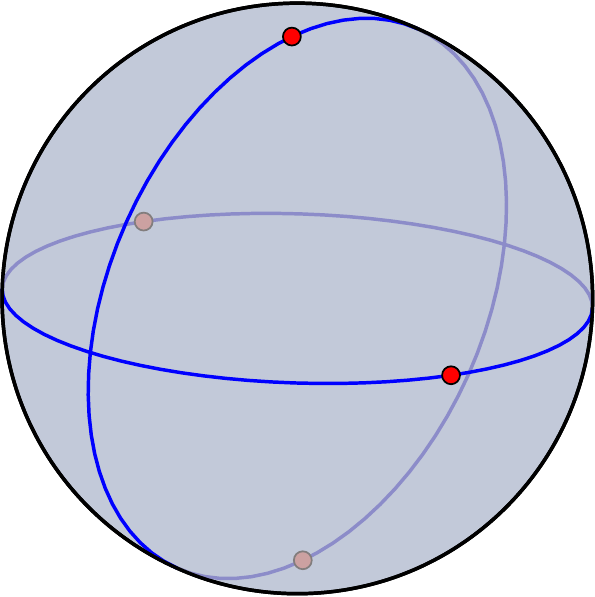}};
\node at (-1.55,0.5){$a$};
\node at (1.3,-0.5){$A$};
\node at (0.1,-0.65){$b$};
\node at (-0.05,2.35){$B$};
\end{tikzpicture}
\caption{Duality between great circles and their poles on the sphere.}\label{flag}
\end{figure}
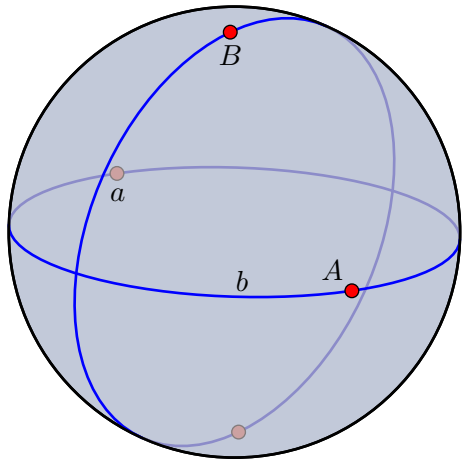

This makes it possible to identify points and great circles, and this duality preserves the incidence relation.

Consider a spherical triangle $ABC$. Let $P$ be the pole of the line $AB$. In terms of the cross-product, one has $P \sim A\times B$, where $\sim$ means that the vectors are proportional. The line $PC$ is orthogonal to $AB$, and the pole of $PC$ is proportional to $P\times C$, see Fig.~\ref{sphereCin}. Thus the altitude dropped from vertex $C$ is given by $(A\times B) \times C$.

\begin{figure}[hbtp]\centering
\begin{tikzpicture}
\node at (0,0){\includegraphics{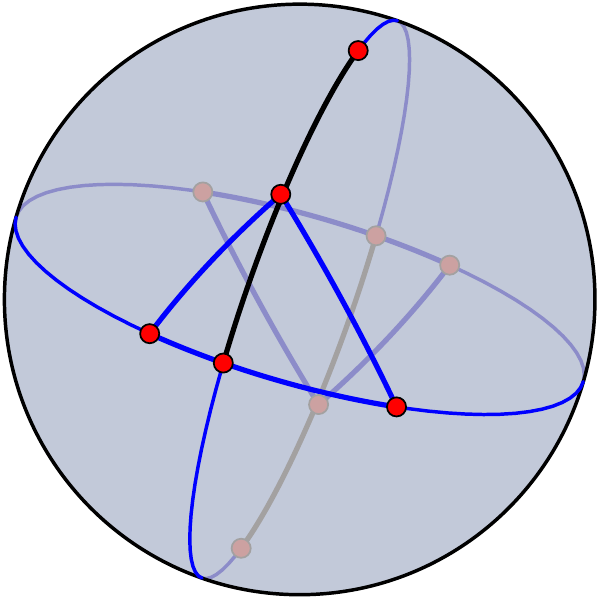}};
\node at (-1.55,0.0){$A$};
\node at (1.1,-0.8){$B$};
\node at (-0.25,1.4){$C$};
\node at (0.3,2.55){$P$};
\end{tikzpicture}
\caption{The altitude as a great circle.}\label{sphereCin}
\end{figure}
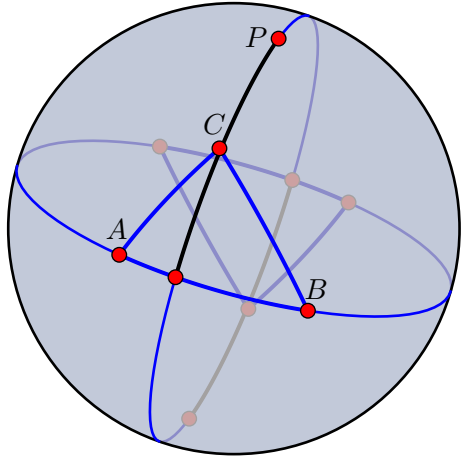

We need to show that the three altitudes are concurrent or, equivalently, that their poles are collinear. That is, we need to show that the three vectors
\[
(A\times B) \times C,\qquad (B\times C) \times A,\qquad (C\times A) \times B
\]
are collinear. And, indeed,{\samepage
\[
(A\times B) \times C + (B\times C) \times A + (C\times A) \times B =0,
\]
the Jacobi identity in the Lie algebra $\big({\mathbb R}^3,\times\big)$, that is, $\mathfrak{so}(3)$.}

A similar argument, due to V. Arnold, applies to triangles in the hyperbolic geometry; the relevant Lie algebra is $\mathfrak{sl}(2,{\mathbb R})$, the Lie algebra of motions of the hyperbolic plane, see~\cite{Ar}. A new point is that the altitudes of a hyperbolic triangle may intersect outside of the hyperbolic plane, in the de Sitter plane.

The intersection of the medians of a spherical triangle follows from the evident formula
\[
(A+B) \times C + (B+ C) \times A + (C+ A) \times B =0,
\]
(why?), which is a manifestation of skew-symmetry of the cross-product. Try to find and explain the corresponding formula for the intersection of the angle bisectors of a spherical triangle.

\subsection{Problem 6: Parking a car}

For a car in the plane the configuration space consists of all quadruples $(x, y, \theta, \varphi) \in
{\mathbb R}^2 \times S^1 \times (-\pi/4, \pi/4)$, where
$(x, y)$ is the position of the midpoint of the rear axle,
$\theta$ is the direction of the car axle,
$\varphi$ is the steering angle of the front wheels (see Fig.~\ref{car} and~\cite{Michor}).

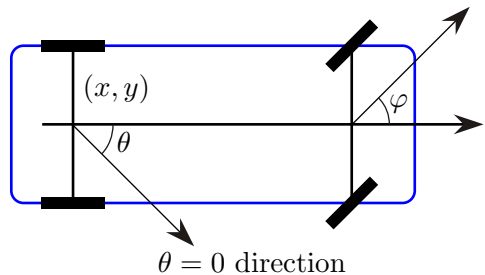
\begin{figure}[hbtp]\centering
\begin{tikzpicture}
\node at (0,0){\includegraphics{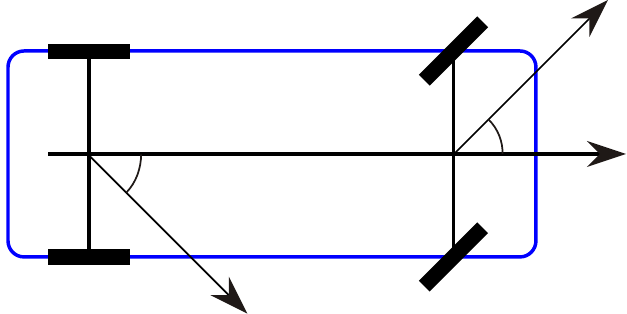}};
\node at (-1.7,0.5){$(x,y)$};
\node at (2.05,0.3){$\varphi$};
\node at (-1.6,-0.2){$\theta$};
\node at (0.1,-1.8){$\theta=0 \ \text{direction}$};
\end{tikzpicture}
\caption{The car position is described by the coordinates $(x,y)$ of its midpoint of the axle,
the angle~$\theta$ of the car axle with a fixed direction, and the steering angle~$\varphi$ of the front wheels, from~\cite{Michor}.}\label{car}
\end{figure}

There are two ``control'' vector fields:
\[
{\rm steer} := \frac{\partial}{\partial \varphi}\qquad{\rm and}\qquad
{\rm drive} := \cos\theta \frac{\partial}{\partial x} +\sin \theta \frac{\partial}{\partial y}+(1/\ell)\tan \varphi \frac{\partial}{\partial\theta},
\]
where $\ell$ is the distance between the front and rear axles. We leave it to the reader to explain these formulas.
Next, compute the vector field
\[
{\rm turn}:=[{\rm steer}, {\rm drive}] = \frac{1}{\ell \cos^2\varphi}\frac{\partial}{\partial\theta}.
\]
Its form explains why it is called this way. Finally, compute the vector field
\[
{\rm park}:=[{\rm drive}, {\rm turn}]= \frac{1}{\ell \cos^2\varphi}\left(\sin\theta \frac{\partial}{\partial x} -\cos \theta \frac{\partial}{\partial y}\right).
\]
Note that the field ${\rm park}$ moves the car perpendicular to the car axis, which is nothing else but parallel parking.

Here are two exercises for the reader:
\begin{enumerate}\itemsep=0pt
\item[$i)$]~Interpret the role of the vector field ${\rm park}$ for different values of $\varphi$;
\item[$ii)$]~Is it to the better or to the worse that the two control vector fields do not span
an integrable distribution? (By the Frobenius theorem, integrability would require the field $[{\rm steer}, {\rm drive}] $
to be a combination of the control fields ${\rm steer}$ and ${\rm drive}$.)
\end{enumerate}

\subsection{Problem 7: Cheating with luggage}

Consider $\varepsilon$-neighborhood of a rectangular box. Its volume is equal to
\[
V + \varepsilon S + 6\pi \varepsilon^2 L + \frac{4}{3}\pi \varepsilon^3,
\]
where $V$ is the volume of the box, $S$ is its surface area, and $L$ is the sum of lengths of its edges.

If the first box sits inside the second one, then the volume of its $\varepsilon$-neighborhood is smaller than that of the second one. Hence
\[
V_1 + \varepsilon S_1 + 6\pi \varepsilon^2 L_1 + \frac{4}{3}\pi \varepsilon^3 <
V_2 + \varepsilon S_2 + 6\pi \varepsilon^2 L_2 + \frac{4}{3}\pi \varepsilon^3.
\]
Cancel the common cubic term, divide by $\varepsilon^2$, and let $\varepsilon \to \infty$ (in contrast with the implicit assumption that $\varepsilon$ is small) to conclude that $L_1 < L_2$.

We refer to \cite{Gr} for more general Steiner and Weyl formulas for volumes of tubular neighborhoods.

\begin{remark*}\label{remark1}
We also note that the surface area of the box inside is also smaller than that of the box that contains it. This follows from the Cauchy--Crofton formula for the surface area.

Indeed, for a convex body $B$, its surface area equals
$(1/2) \mu$, where $\mu$ is the volume of the set of the oriented lines that intersect~$B$.
The space of oriented lines in ${\mathbb R}^n$ is identified with the cotangent bundle $T^* S^{n-1}$, and the volume form is induced by the canonical symplectic structure of the cotangent bundle, see, e.g.,~\cite{Sa}.
Every line that intersects the inner box also intersects the one that contains it, and this implies the result.
\end{remark*}

\subsection{Problem 8: How to be 1/3 Spanish?}

V.~Arnold discussed in \cite{Arnold} the following problem of practical origin.
A river has a flux of water equal to one (for instance, 1 cubic meter of water per second). Let us call a device that can
split the flow into two equal parts a simple divider. So by applying such a simple divider to the river, one can fork it into two smaller streams
with fluxes 1/2 each. As an example, one can think of a crowd of people approaching a turnstile which sends every other person to the right and the next one to the left, thus splitting the queue into two equal ones. Then, by applying two such dividers consecutively, one can separate 1/4 of the initial flow,
see Fig.~\ref{divider}.

Here is the main question: can one use a finite number of such simple dividers to separate~1/3 of the flow?

\begin{figure}[hbtp]\centering
\begin{tikzpicture}
\node at (0,0){\includegraphics{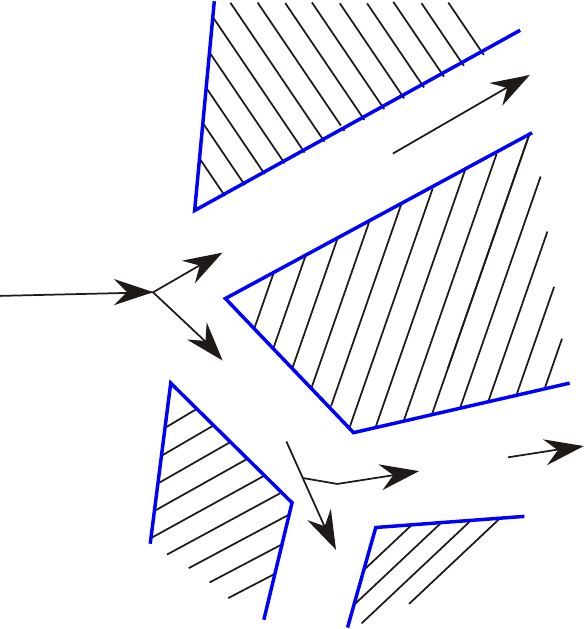}};
\node at (-3.1,0.19){$1$};
\node at (0.1,1.1){$\dfrac 12$};
\node at (-0.3,-0.8){$\dfrac 12$};
\node at (1.6,-1.5){$\dfrac 14$};
\node at (0.25,-2.8){$\dfrac 14$};
\end{tikzpicture}
\caption{Splitting flow in two, and then again in two, see~\cite{Arnold}.}\label{divider}
\end{figure}
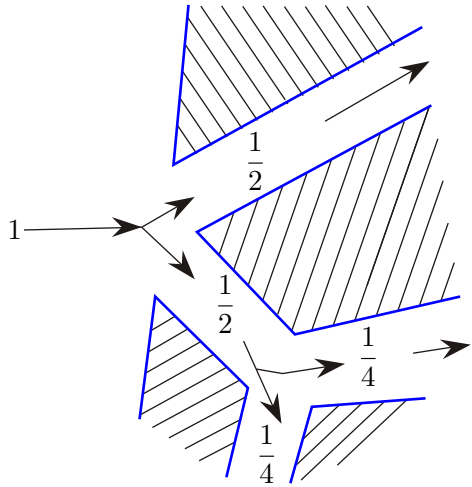

Arnold writes that this problem arose in the civil defense drills in the Moscow subway:
they tried to direct the crowd of people into three nuclear bomb shelters of equal capacity!

This problem reminds a question we discussed in our high schools: ``Can one be 1/3 Spanish?'' Indeed, when only one of two parents is Spanish, it makes their child 1/2 Spanish. If only one of his or her four grand-parents is Spanish, the child is 1/4 Spanish. But this is exactly the simple divider problem!
So the question on the possibility of being 1/3 Spanish is the problem of separating~1/3 of the flow.
(Some people discuss an asymptotic solution, or ``convergence to~1/3'', but in those terms it would mean the use of an infinite number of dividers, and it is discarded here.) There is the tongue-in-cheek answer to this ``Spanishness'' question: ``Yes, one can be 1/3 Spanish, if both one's father is 1/3 Spanish and one's mother is 1/3 Spanish!''

Curiously, this question has yet another handy reformulation. Three thieves would like to split their
loot, item by item, via a lottery with equal probability 1/3 of winning for each of them with the help of a standard ``heads-or-tails'' coin. How should they organize their lottery?

Have you noticed the analogy with splitting the river into three equal flows?

The answer to the flow question is surprisingly simple. Let us split the river into the four equal streams, but then divert one of them back into the mainstream. The remaining three streams are completely equal and each one will carry 1/3 of the flux, see Fig.~\ref{diverter}.
(Actually, such a solution works for any $p/q$ part of the original flow: split the flow into $2^n\ge q$ streams, and then divert back whatever is not needed.)

\begin{figure}[hbtp]\centering
\begin{tikzpicture}
\node at (0,0){\includegraphics{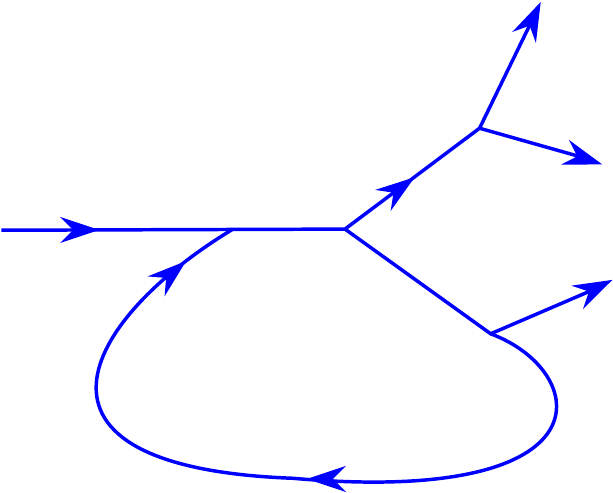}};
\node at (-2.3,0.45){$1$};
\node at (-0.2,0.65){$\dfrac 43$};
\node at (0.8,1.15){$\dfrac 23$};
\node at (1.25,0.05){$\dfrac 23$};
\node at (2.6,2.4){$\dfrac 13$};
\node at (3.2,0.9){$\dfrac 13$};
\node at (3.3,-0.3){$\dfrac 13$};
\node at (0.2,-1.8){$\dfrac 13$};
\end{tikzpicture}
\caption{Divert one of four streams back into the main one.}\label{diverter}
\end{figure}
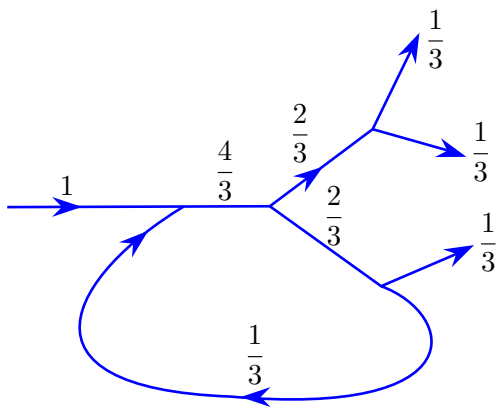

\looseness=-1 For the thieves this would mean the following algorithm for each item: start by tossing the coin two times. For the outcome $(H, H)$ the thief $A$ wins this item, for $(H,T)$ the thief $B$ wins~it, $(T,H)$ the thief $C$ wins it, while if the outcome is $(T,T)$, start the process of tossing for this item anew.

One may argue that the infinity has reentered here: one does not have a uniform bound on the number of coin tosses.
 For the shelter setting this
would mean that for some people it will take arbitrarily long to get to the shelters.
And we do not even dare to say, what kind of inbreeding this means for being 1/3 Spanish \dots.

\subsection{Problem 9: Hunter's tent}

One evident solution is the North Pole $NP$. But it is not unique!

Consider the circle of latitude $\ell_1$ of length exactly 10 km near the South Pole. Then the hunter's tent can be {\it anywhere} on the circle of latitude $m_1$ that is 10 km north of $\ell_1$.
Furthermore, for the parallel $\ell_2$ of length equal to 5 km in the Southern hemisphere, one can consider the parallel $m_2$ that is 10 km north of it. In general, the tent can stand at any point of the infinite number of circles $m_k$, $k=1,2,3,\dots $ that are 10 km north of parallels $\ell_k$ of length $10/k$ km in the vicinity of the South Pole, see Fig.~\ref{tents}. Note that this set of all solutions
\[
NP\cup \{ m_k \,|\, k=1,2,3,\dots\}
\]
 is not closed!

\begin{figure}[hbtp]\centering
\begin{tikzpicture}
\node at (0,0){\includegraphics{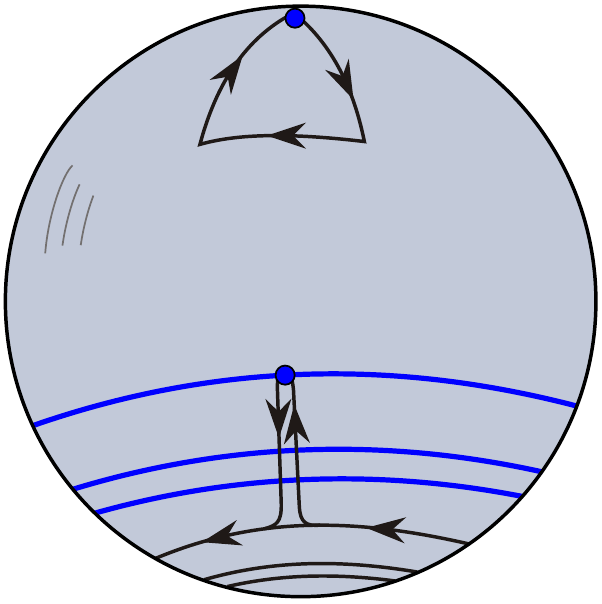}};
\node at (0,3.2){$NP$};
\node at (-3.0,-1.3){$m_1$};
\node at (-2.7,-1.9){$m_2$};
\node at (-2.4,-2.3){$m_3$};
\node at (1.9,-2.6){$\ell_1$};
\node at (1.45,-2.85){$\ell_2$};
\node at (1.1,-3.0){$\ell_3$};
\node at (0,-3.2){$SP$};
\end{tikzpicture}
\caption{All possible locations of the hunter's tent.}\label{tents}
\end{figure}
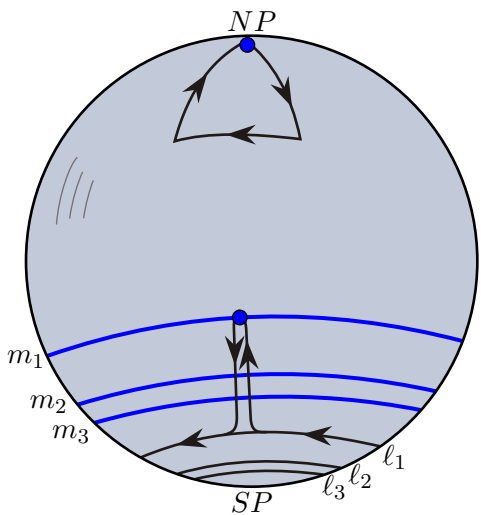

What would you answer if asked what animals the hunter was hunting for, cf.~\cite{Ar-kids-problems}?

\subsection{Problem 10: Two boys in a family?}

The seemingly evident answer 1/2 is correct, but for a different problem:
``It is known that the first of the kids is a~boy born on Friday. What is the probability that the second child is also a~boy?''
However, the actual formulation is different and has a surprising answer: the probability that the other child is also a boy is 13/27.

This problem, apparently, goes back to Martin Gardner with an extra twist added by Gary Foshee. It turns out that the answer strongly depends on the procedure of obtaining the information about families with kids, see an interesting discussion of possible cases in~\cite{Khovanova}. Here we accept the so called ``Boy-centered, Friday-centered'' procedure, and discuss the problem in the most naive way.

To start, let us forget about Friday, and first solve original Garder's problem:
``It is known that one of the kids is a boy. What is the probability that the other child is also a boy?''

Then a simple $2\times 2$ table for the first (rows $B$ and~$G$) and the second (columns $B$ and $G$) kids gives the following ordered pairs: $BB$, $BG$, $GB$, and $GG$, where~$B$ stands for a~boy and~$G$ for a girl, and equal probability for each of them.
The fact that one of the kids is a~boy rules out the~$GG$ case. We are left with three options,
$BB$, $BG$, $GB$, and we are interested in the probability of~$BB$ (i.e., two boys in the family) among these three.
Thus, the required probability in this simplified problem is~1/3.

Now return to the ``boy-on-Friday'' problem. Now we have not $2\times 2$, but $14\times 14$ table, see Fig.~\ref{boys-table}:
 for each, the first and the second child, we have 14 options. Namely, the first child could be a boy born on Monday ($B_M$), a boy born on Tuesday ($B_T$), etc, or it could be a girl born on Monday ($G_M$), on Tuesday ($G_T$), etc. Similarly, we have 14 options for the second child.

\begin{figure}[hbtp]\centering
\begin{tikzpicture}
\node at (0,0){\includegraphics[height=3in]{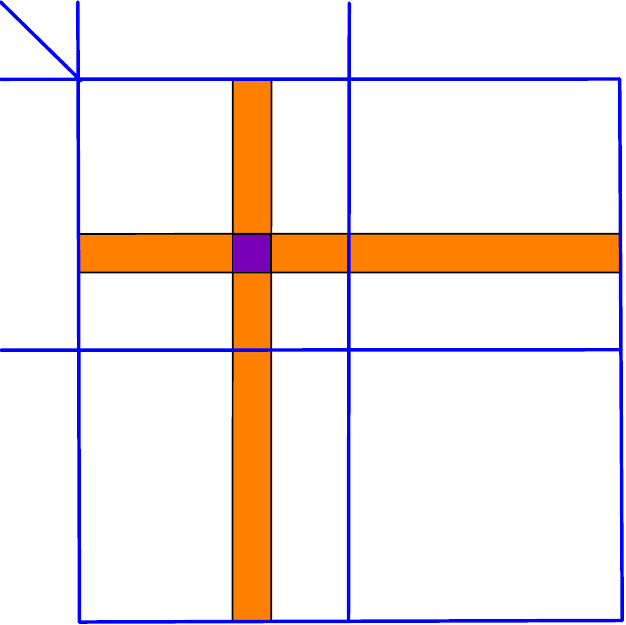}};
\node at (-3.5,3.2){${\rm I}$};
\node at (-3.1,3.5){${\rm II}$};
\node at (-2.6,3.2){${\rm M}$};
\node at (-2.2,3.2){${\rm T}$};
\node at (-1.8,3.2){${\rm W}$};
\node at (-1.3,3.2){${\rm Th}$};
\node at (-0.7,3.2){${\rm F}$};
\node at (-0.3,3.2){${\rm Sa}$};
\node at (0.15,3.2){${\rm Su}$};
\node at (-1.3,3.7){${\rm B}$};

\node at (0.7,3.2){${\rm M}$};
\node at (1.1,3.2){${\rm T}$};
\node at (1.5,3.2){${\rm W}$};
\node at (2.0,3.2){${\rm Th}$};
\node at (2.45,3.2){${\rm F}$};
\node at (2.9,3.2){${\rm Sa}$};
\node at (3.4,3.2){${\rm Su}$};
\node at (1.6,3.7){${\rm G}$};

\node at (-3.3,2.6){${\rm M}$};
\node at (-3.3,2.15){${\rm T}$};
\node at (-3.3,1.7){${\rm W}$};
\node at (-3.3,1.25){${\rm Th}$};
\node at (-3.3,0.75){${\rm F}$};
\node at (-3.3,0.30){${\rm Sa}$};
\node at (-3.3,-0.15){${\rm Su}$};
\node at (-3.9,1.25){${\rm B}$};

\node at (-3.3,-0.7){${\rm M}$};
\node at (-3.3,-1.15){${\rm T}$};
\node at (-3.3,-1.60){${\rm W}$};
\node at (-3.3,-2.05){${\rm Th}$};
\node at (-3.3,-2.55){${\rm F}$};
\node at (-3.3,-3.00){${\rm Sa}$};
\node at (-3.3,-3.45){${\rm Su}$};
\node at (-3.9,-2.05){${\rm G}$};

\end{tikzpicture}
\caption{$14\times 14$ options for the pair: (first child, second child).}\label{boys-table}
\end{figure}

The information given, that one of the kids is a boy born on Friday, singles out one row and one column:
$(B_F, *)$ and $(*,B_F)$, which intersect over one box $(B_F,B_F)$.
Altogether we have 27 boxes in these row and column. And we are interested only in the part lying in the $BB$ part, i.e.,
in $(B_F,B_*)$ or $(B_*,B_F)$, which contains of 13 small boxes, because of the one intersection $(B_F,B_F)$ discussed above.
Thus the required probability is~13/27. Note, that this probability is closer to 1/2 than the previously obtained~1/3.

If we know also the hour when one of the kids was born, we have to split each column and each row according to the~24 hours in the day. The corresponding probability will be 335/671, even closer to 1/2 than 13/27.
This happens because the more information we know about one of the kids, the smaller is the probability that
the other child was born on the same weekday, at the same hour, etc. Hence the smaller is the intersection of the row and the column, which is ``responsible'' for the difference from 1/2.

\subsection{Problem 11: Setting a table on an uneven floor}

By a table we mean four points, the tips of the legs, and we ignore a more complex structure of a real table (such as the length of the legs, the existence of the table top, etc.) There are three levels of sophistication in this problem. The easiest case is that of a square table.

Assume that the floor is a graph of a `reasonable' function, and the floor is transparent to the tips of the table's legs.
Given a square $ABCD$, place the diagonal points $A$ and $C$ on the floor, and rotate the square in space about the line $AC$ so that the signed vertical distances from points~$B$ and~$D$ to the floor be equal (the sign is positive if the points are above, and negative if they are below the floor). We assume that once the diagonal $AC$ is fixed on the floor, such position of the square $ABCD$ is unique (this is an assumption on the shape of the floor: it should not be too `wild'). Call this ``position 1''.

Parallel translate the square in the vertical direction so that points $B$ and $D$ are now on the floor. Then points $A$ and $C$ are at equal distance from the floor, the same distance as in position~1, but with the opposite sign. Call this position of the table ``position~2''.

Now continuously move the table from position 1 to position 2, keeping points $A$ and $C$ on the floor (for example, by rotating the table through about~$90^{\circ}$). In the process, the distance from points $B$ and $D$ to the floor changes continuously and it changes the sign. Hence there was a moment when the distance equaled zero, and all four legs were set on the floor.

This problem and its solution is part of mathematical folklore. In particular, it was presented by M. Gardner in his Scientific American ``Mathematical Games'' column in 1973. A slightly different proof is described by M.~Kreck in the Numberphile video \url{https://www.youtube.com/watch?v=OuF-WB7mD6k}.

The next level of sophistication is when the table is rectangular. Gardner says, without explanation, that a similar argument works, but the rotation should be not through $90^{\circ}$, but through $180^{\circ}$. This 1-dimensional, intermediate value theorem, argument, is presented in \cite{BLPR} and in the video \url{https://www.youtube.com/watch?v=aCj3qfQ68m0}. We must confess that we do not understand this argument. However, the result holds: one can set a rectangular table of any aspect ratio on an uneven floor.

The proof is also presented in \cite{BLPR}, but it makes use of a serious topological result by Live\-say~\cite{Li}. This theorem states that, given a continuous function $F$ on the 2-dimensional sphere, and an angle $\alpha$, there exist two diameters of the sphere making angle $\alpha$ such that $F$ takes the same value at the end points of these diameters.

To apply this theorem to a rectangular table, assume that the center of the rectangle (the table) is confined to a vertical axis. For a given position of this center, the four vertices lie on a sphere, and the vertical distance from the points of this sphere to the floor is a continuous function. Thus this function assumes the same value at the four vertices of the rectangle, and parallel translating the table vertically places its four legs on the floor.

The third level in this problem is a conjecture: a four-legged table can be set on an uneven floor if and only if the four legs form an planar inscribed quadrilateral. The condition is necessary, otherwise the table cannot be placed on a plane or a sphere.

This problem is loosely analogous to the Toeplitz, or square peg, conjecture that every Jordan curve has an inscribed square. This conjecture is proved in particular cases (polygons, smooth curves), but it remains open in the full generality; see~\cite{Ma} for a recent survey. However, a~stronger result holds for an arbitrary smooth convex curve $\gamma$: there exists a homothetic copy of an arbitrary cyclic quadrilateral inscribed in~$\gamma$, see~\cite{AA}.

Let us also mention a variant of the problem when the floor has the shape of a~hill with a~convex boundary, and one wants to place the square table so that the tabletop is horizontal, see~\cite{Fe,KK,Me}.

\subsection{Problem 12: A conic mountain}

The rope will tighten on the mountain along a geodesic. Geodesics on a cone are straight lines on its unfolding, which is a sector
of some angle. Make this unfolding by cutting the cone
along the ray passing through the knot of the loop, so that the straight line goes from one sector's side to the other.
The loop will slide away if the straight line joining the sector sides does not lie within the sector, i.e., if the sector's angle
is at least $ \pi$.

What cone does the sector of angle $\pi$ correspond to, i.e., what cone can be bend from half a~disk?
If the radius of the sector is~$R$, its arc length is $\pi R$ (half a circumference). Once we bend the cone,
this arc will become the circumference of radius~$R/2$. Thus the cone has the generator of length $R$, while its base is a circle of radius~$R/2$. These are respectively the hypotenuse and the leg of the right triangle spanning the cone. This spanning triangle has the angle $\pi/6$ at the cone vertex.

\subsection{Problem 13: Two nested ovals}

Let $AB$ be a chord of the outer oval, $\Gamma$, tangent to the inner one, $\gamma$, and let $A'B'$ be its perturbation, see Fig.~\ref{ovals}. Let $\varepsilon$ be the angle between the chords. Then the difference of the areas of the infinitesimal triangles $AOA'$ and $BOB'$ is
\[
\frac{1}{2}\varepsilon \big(|AO|^2-|BO|^2\big).
\]

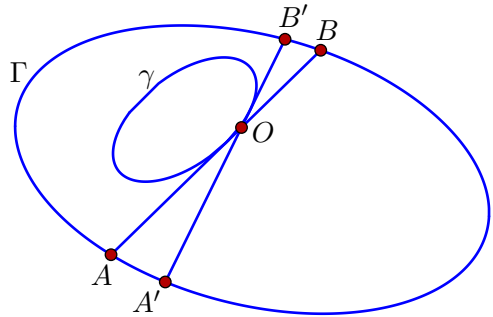
\begin{figure}[hbtp]\centering
\begin{tikzpicture}
\node at (0,0){\includegraphics{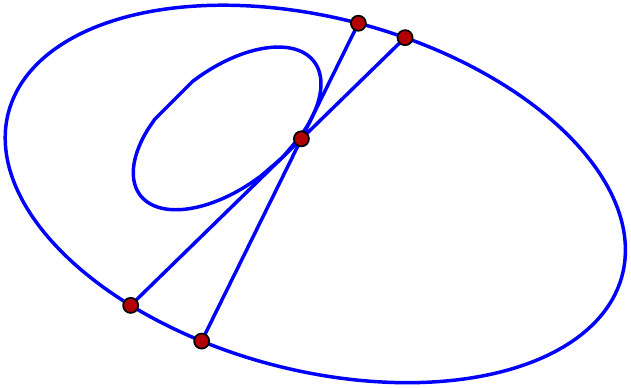}};
\node at (-3.1,1.3){$\Gamma$};
\node at (-1.4,1.2){$\gamma$};
\node at (0.15,0.5){$O$};
\node at (0.55,2.05){$B'$};
\node at (1.00,1.85){$B$};
\node at (-1.40,-1.75){$A'$};
\node at (-2.00,-1.40){$A$};
\end{tikzpicture}
\caption{A chord and its perturbation tangent to the inner oval.}\label{ovals}
\end{figure}

This implies that $|AO|=|BO|$ if and only if the area on one side of the chord $AB$ has an extremal value. Since this area attains a maximum and a minimum, there are at least two chords that are bisected by the tangency point.

A somewhat different approach is to consider the outer billiard map about the inner oval $\gamma$. Given a point $x$ in the exterior of $\gamma$, draw a tangent line to $\gamma$ and reflect $x$ in the tangency point to obtain a new point $y$. The map $x \mapsto y$ is the outer billiard map.

The outer billiard map is area preserving, see, e.g., \cite{TaB,DT}. Hence the image of the outer oval $\Gamma$ must intersect $\Gamma$ at least two points. This yields two desired chords.

\subsection{Problem 14: Wrap a string}

The proof is illustrated in Fig.~\ref{string}. Let $O$ be a point on the ``invisible'' from point $X$ side of the given curve. Consider two functions of point $X$: $f(X)$ is the distance to $O$ going on the right of the curve, and $g(X)$ is the distance to $O$ going on the left of the curve.

The gradients of these functions are unit vectors along the lines $AX$ and $BX$, respectively.
Indeed, the level curve of $f(X)$ is an involute of the curve, and its tangent vector is orthogonal to the line $AX$. Hence $\nabla f(X)$ is collinear with $AX$, and the fact that this vector is unit is tautological. Likewise for $\nabla g(X)$.

\begin{figure}[hbtp]\centering
\begin{tikzpicture}
\node at (0,0){\includegraphics{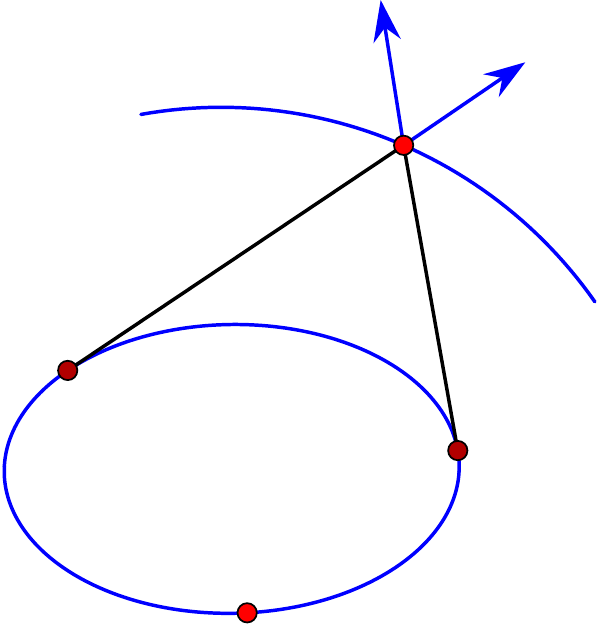}};
\node at (-0.5,-2.75){$O$};
\node at (-2.4,-0.3){$A$};
\node at (1.9,-1.4){$B$};
\node at (1.2,2.10){$X$};
\end{tikzpicture}
\caption{Wrapping a string around an oval.}\label{string}
\end{figure}
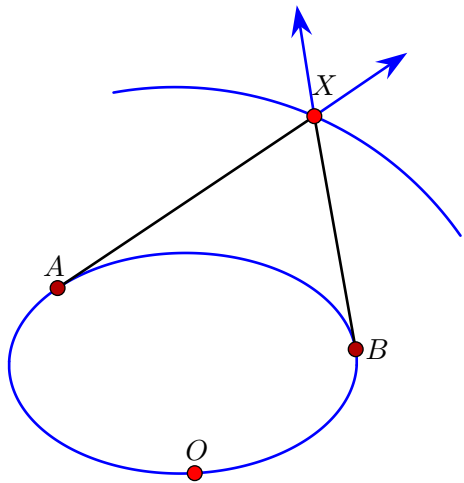

The result of the string construction is the curve given by the equation $f(X)+g(X)=$ const. The gradient of the function $f+g$ is perpendicular to this curve, and this gradient is the sum of two unit vectors along the lines $AX$ and $BX$. Hence these lines make equal angles with the curve.

\subsection{Problem 15: Ivory's lemma}

About 200 years ago, J.~Ivory studied the gravitational potential of a homeoid, the layer between an ellipsoid and a homothetic ellipsoid. He proved that the equipotential surface outside of the homeoid are the confocal ellipsoids. The main geometric ingredient in the proof was Ivory's lemma, illustrated in Fig.~\ref{Ivory} in its simplest case of the Euclidean plane.

\begin{figure}[hbtp]\centering
\includegraphics{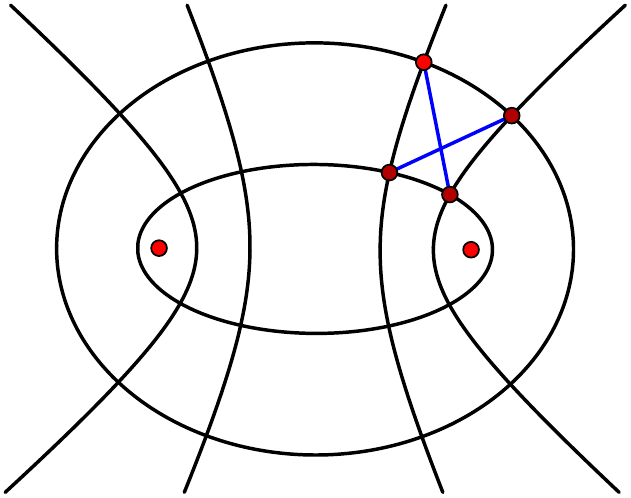}
\caption{Ivory's lemma: the diagonals of a curved quadrilateral are equal.}\label{Ivory}
\end{figure}

We shall deduce Ivory's lemma from integrability of billiards in ellipses. See \cite{AB,Ar-problems,GIT, IT} for these and related topics, also including the material of the next section.

A caustic of a billiard is a curve such that if a billiard trajectory is tangent to this curve, then the reflected trajectory is again tangent to it. The billiard in an ellipse admits a family of caustics, the confocal ellipses. Therefore the string construction on an ellipse yields a confocal ellipse (the Graves theorem).

The Arnold--Liouville theorem from the theory of integrable systems implies that each ellipse carries a coordinate in which the billiard reflection in a confocal ellipse is given by a shift; furthermore, the same coordinate describes the reflections in all confocal ellipses (the value of the shift of course depends on the choice of these ellipses). In terms of Fig.~\ref{coord}, the map $A\mapsto B$ is given by $s\mapsto t= s+c$, where~$s$ is this special parameter on the ellipse.

\begin{figure}[hbtp]\centering
\begin{tikzpicture}
\node at (0,0){\includegraphics{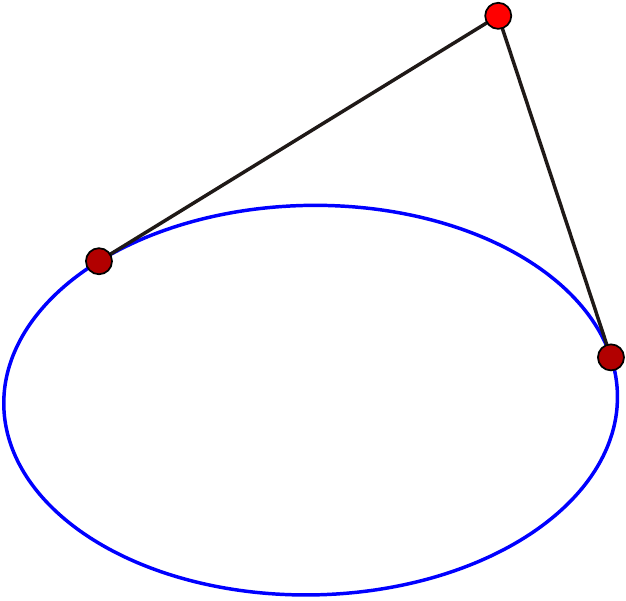}};
\node at (2.6,2.9){$C(s,t)$};
\node at (-2.7,0.75){$A=\gamma(s)$};
\node at (3.9,-0.6){$B=\gamma(t)$};
\end{tikzpicture}
\caption{The map $A\mapsto B$.}\label{coord}
\end{figure}
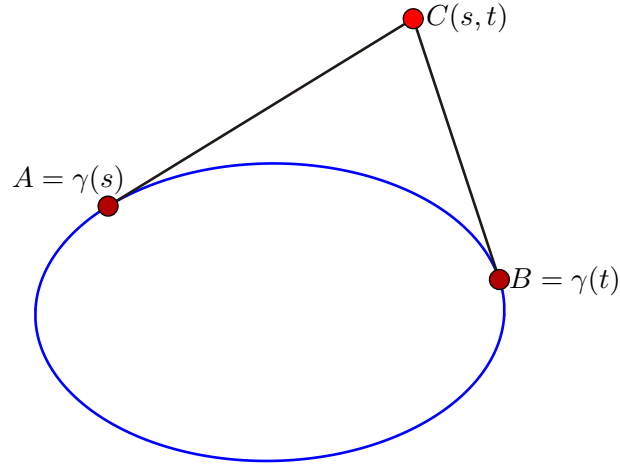

Let $\gamma(t)$ be an ellipse with this special parameterization. Introduce coordinates in its exterior by drawing the tangent lines to $\gamma$ and taking the parameters of the tangency points, see Fig.~\ref{coord}. The confocal ellipses are given by the equation $t-s= {\rm const}$, and the confocal hyperbolas by $t+s= {\rm const}$. That is, the billiard reflection in a confocal ellipse is given by $t\mapsto t+c$, and in a~confocal hyperbola by $t\mapsto c-t$; the constant $c$ depends on the choice of the ``mirror''.

The idea of the proof is illustrated by the following argument.

Suppose that one wants to prove that the diagonals of a rectangle have equal length by using billiards. First we note that the billiard reflection is well defined in a right angle: a close parallel trajectory exists the angle in the opposite direction after two reflections. Therefore one may consider a diagonal as a 4-periodic trajectory. It is included into a~1-parameter family of 4-periodic trajectories that interpolate between the two diagonals, see Fig.~\ref{rectangle} on the left.

It remains to note that periodic trajectories in a 1-parameter family have the same length: $n$-periodic trajectories are critical points of the perimeter function on inscribed $n$-gons, and a~function has constant value on its critical manifold.

\begin{figure}[hbtp]\centering
\includegraphics{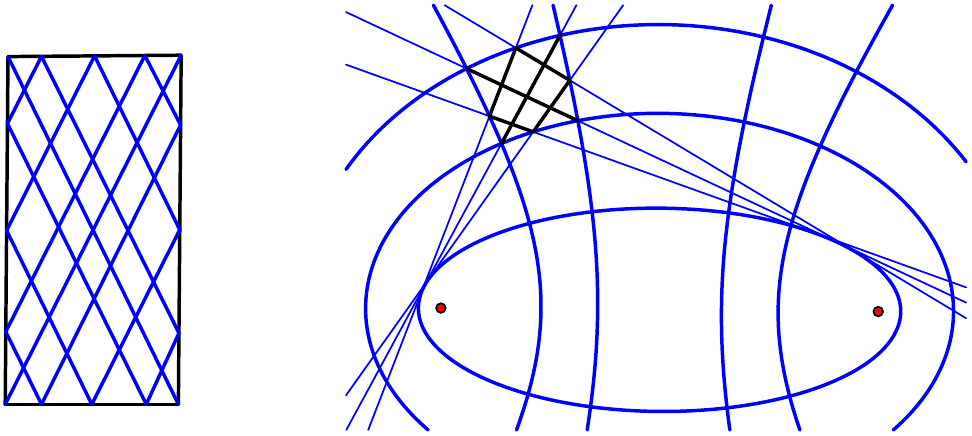}
\caption{The diagonal is included into a family of 4-periodic billiard trajectories.}\label{rectangle}
\end{figure}

The same argument works for Ivory's lemma. Confocal ellipses and hyperbolas are orthogonal, hence we may consider a diagonal as a 4-periodic billiard trajectory in a curvilinear rectangle. This trajectory is included into a 1-parameter family of 4-periodic trajectories. Indeed, in the special parameter on the caustic, the reflections are given by
\[
t \mapsto t + a,\qquad t\mapsto b-t, \qquad t \mapsto t + c,\qquad t\mapsto d-t.
\]
The composition of these maps is a shift. Since it has a fixed point (the initial diagonal), it is the identity. The rest of the argument is the same as for rectangles.

Ivory's lemma holds in higher dimensions as well and in other Riemannian metrics. For example, it holds on the surface of an ellipsoid in ${\mathbb R}^3$ with the lines of curvature playing the roles of confocal conics, see Fig.~\ref{ellipsoid}.

\begin{figure}[hbtp]\centering
\includegraphics[height=1.4in]{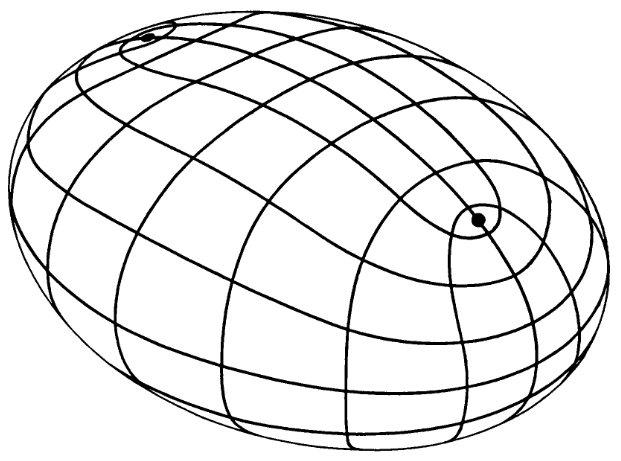}
\caption{Ivory's lemma on an ellipsoid in ${\mathbb R}^3$ with the lines of curvature as confocal conics.}\label{ellipsoid}
\end{figure}

In fact, the Ivory property of a net of curves on a Riemannian surface is equivalent to the property of the metric to be Liouville, that is, to have the form
\[
(f_1(x) - g_1(y)) \big(f_2(x) {\rm d}x^2 + g_2(y) {\rm d}y^2\big)
\]
in local coordinates $(x,y)$, and the curve to be the coordinate lines $x=$ const, $y=$ const (a~theorem of Blaschke).

\subsection{Problem 16: Chasles--Reye theorem}

The proof uses the special coordinate on the inner ellipse that we discussed above, see Fig.~\ref{Chasles1}.

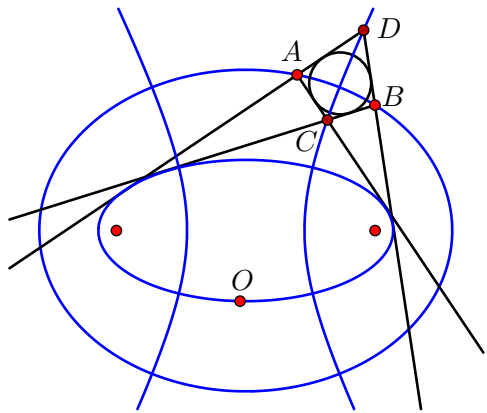
\begin{figure}[hbtp]\centering
\begin{tikzpicture}
\node at (0,0){\includegraphics{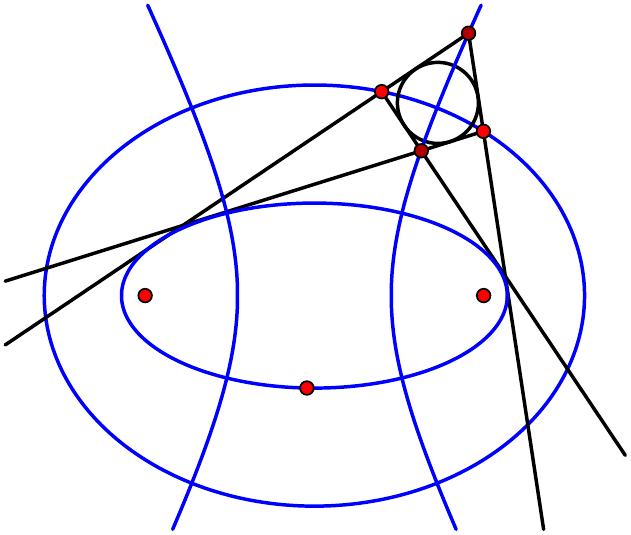}};
\node at (1.9,2.4){$D$};
\node at (0.6,2.1){$A$};
\node at (1.95,1.5){$B$};
\node at (0.8,0.9){$C$};
\node at (-0.05,-0.95){$O$};
\end{tikzpicture}
\caption{Proof of the Chasles--Reye theorem.}\label{Chasles1}
\end{figure}

Let $s_1$, $s_2$, $t_1$, $t_2$ be the coordinates of the tangency points of the lines $AD$, $BC$, $AC$, $BD$ with the ellipse. Then
$t_1-s_1=t_2-s_2$, hence $t_1+s_2=t_2+s_1$, and therefore points $C$ and $D$ lie on a confocal hyperbola.

Let $f$ and $g$ be the distances from points to point $O$ going around the ellipse. Then
\[
f(A)+g(A)=f(B)+g(B),\qquad f(C)-g(C)=f(D)-g(D),
\]
hence
\[
f(D)-f(A) - g(A)+g(C)+f(B)-f(C)-g(D)+g(B)=0,
\]
or
\[
|AD| + |BC| = |AC| + |BD|.
\]
This is a necessary and sufficient condition for a quadrilateral to be circumscribed.

\subsection{Problem 17: Equitangent problem}

The problem can be reformulated as follows: is it possible to move around a chord $AB$ inside an oval so that $\angle ABC > \angle BAC$ holds during the process (see Fig.~\ref{chord})?

\begin{figure}[hbtp]\centering
\begin{tikzpicture}
\node at (0,0){\includegraphics{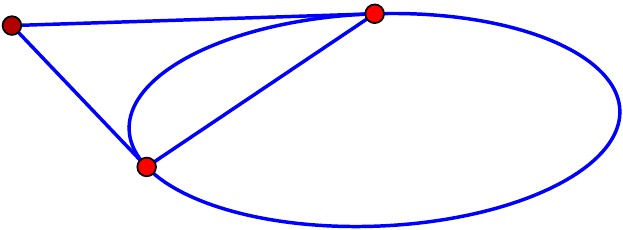}};
\node at (0.65,1.3){$A$};
\node at (-1.7,-0.8){$B$};
\node at (-3.3,0.95){$C$};
\end{tikzpicture}
\caption{Is it possible to move around a chord $AB$ inside an oval so that $\angle ABC > \angle BAC$?}\label{chord}
\end{figure}
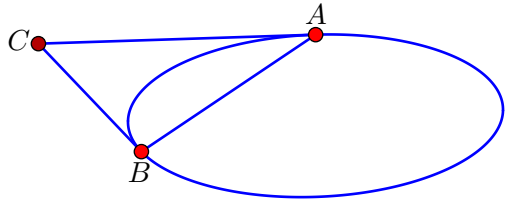

What follows is a piecewise linear version of the construction: the oval is replaced by a convex polygon, the chord is described by its endpoints and the directions of the support lines at its endpoints. One may think about an axle of variable length with two wheels that can steer independently of each other.
The motion is also piecewise linear: at each step, one of the four elements (an endpoint or a support direction) varies linearly, see Fig.~\ref{hex}. A smoothing of the polygon yields a smooth strictly convex curve with the desired property.

\begin{figure}[hbtp]\centering
\begin{tikzpicture}
\node at (0,0){\includegraphics{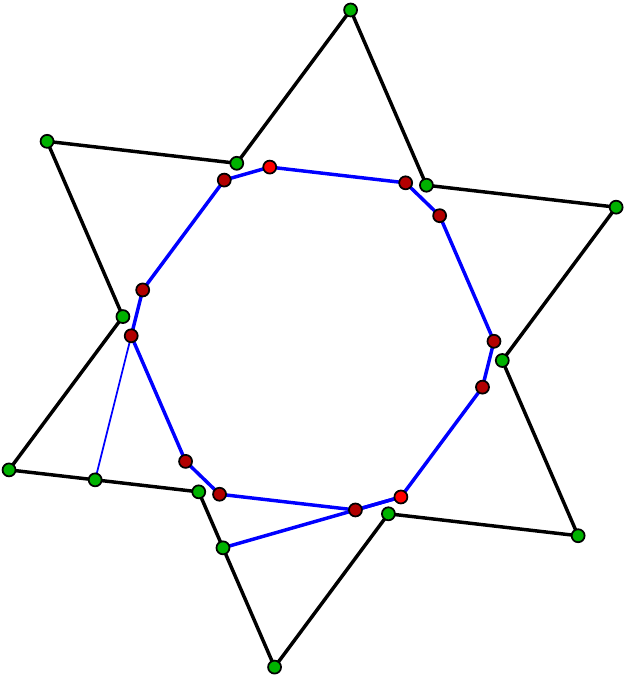}};
\node at (-0.9,1.3){$1$};
\node at (-1.5,0.5){$2$};
\node at (-1.6,0.05){$3$};
\node at (-1.2,-1.0){$4$};
\node at (-0.85,-1.3){$5$};
\node at (0.4,-1.5){$6$};
\node at (0.8,-1.4){$7$};
\node at (1.5,-0.45){$8$};
\end{tikzpicture}
\caption{The piecewise linear motion of the chord.}\label{hex}
\end{figure}
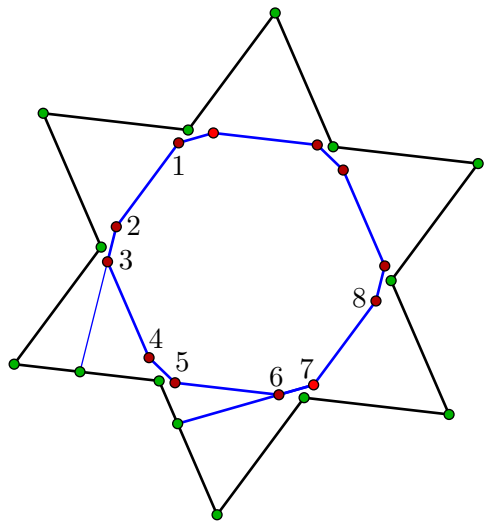

The dodecagon has a six-fold rotational symmetry, the star of David is the locus of the intersection points of the support lines at the endpoints of the moving chord (point $C$ in Fig.~\ref{chord}). Due to the six-fold symmetry, we describe only $1/6$ of the whole process.

The starting chord is $15$ with the support directions at its endpoints $12$ and $56$. The terminal chord is $37$ with the support directions $34$ and $78$. The terminal chord differs from the starting one by rotation through $\pi/3$.

Here is the whole process:
\begin{gather*}
 (15,12,56) \mapsto (25,12,56) \mapsto (25,23,56) \mapsto (35,23,56) \mapsto\\
 (35,34,56) \mapsto
(36,34,56) \mapsto (36,34,67) \mapsto (37, 34,67) \mapsto (37,34,78).
\end{gather*}

See \cite{Ta} for more details and comments. We mention a relevant result from~\cite{RMBC}, motivated by the flotation theory: given an oval~$\gamma$ and a fixed angle~$\alpha$, the locus of points from which the oval is seen under angle $\alpha$ contains at least four points from which the tangent segments to~$\gamma$ have equal lengths.

\subsection{Problem 18: Pentagons in the projective plane}

Both assertions are illustrated in Fig.~\ref{pentagon}.

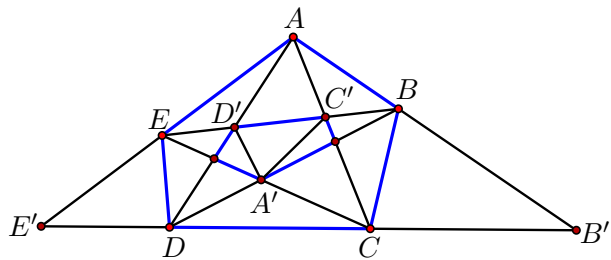
\begin{figure}[hbtp]\centering
\begin{tikzpicture}
\node at (0,0){\includegraphics[scale=1.2]{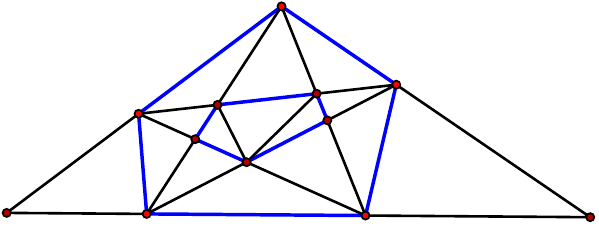}};
\node at (-0.2,1.55){$A$};
\node at (-2.0,0.2){$E$};
\node at (-1.1,0.3){$D'$};
\node at (0.4,0.5){$C'$};
\node at (1.3,0.6){$B$};
\node at (-3.8,-1.2){$E'$};
\node at (-1.8,-1.5){$D$};
\node at (-0.6,-0.9){$A'$};
\node at (0.8,-1.5){$C$};
\node at (3.8,-1.3){$B'$};
\end{tikzpicture}
\caption{The inner pentagon, obtained by the pentagram map, has the same cross-ratio at the vertex~$A'$ as the original large pentagon had at the vertex~$A$.}\label{pentagon}
\end{figure}

At every vertex of the pentagon one has four lines: two adjacent sides and two adjacent diagonals. These four concurrent lines have the cross-ratio, and these five cross-ratios projectively determine the pentagon. These five numbers are not independent -- the moduli space of projective pentagons is 2-dimensional -- and the relations between these numbers obey the rules of the theory of cluster algebras (these relations were calculated by Gauss in a different setting in his unpublished paper ``Pentagramma Mirificum'').

Consider the image of the pentagram map $T$, the inner pentagon. One has
\[
[AB, AC, AD,AE] = [B,C',D',E] = [A'B, A'C', A'D',A'E].
\]
It follows that there exists a projective transformation that takes point $A$ to $A'$, and other vertices of the initial pentagon to the opposite vertices of its image under the pentagram map.

Let $P=ABCDE$. Then $[AB, AC, AD,AE] = [B',C,D,E']$. The lines, dual to these four points, pass through the point, dual to the line $CD$, that is, through the vertex~$A^*$ of the dual pentagon. The lines $C^*$ and $D^*$ are the sides of the dual pentagon, and the lines $(B')^*$ and $(E')^*$ are its diagonals.

We conclude that $P$ is projectively equivalent to $P^*$, namely, there exists a projective transformation that takes~$A$ to $A^*$, $B$ to $B^*$, etc., see~\cite{FT} for more detail.

Back to the pentagram map: it is an involution on projective hexagons, but we do not know an elementary proof of this fact.

We should add that the Poncelet polygons, that is, polygons that are inscribed into a conic and circumscribed about a conic, are self-dual, and that the image of a Poncelet polygon $P$ under the pentagram map $T$ is projectively equivalent to~$P$ (every pentagon is a~Poncelet polygon). Both facts follow from the Poncelet grid theorem \cite{Sch}.

\begin{figure}[hbtp]\centering
\includegraphics{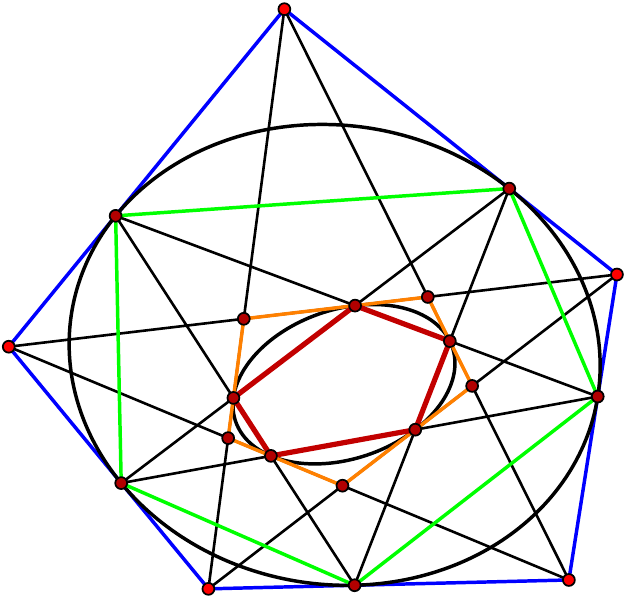}
\caption{The pentagram map $T$ commutes with the operation $I$ taking the vertices of a pentagon to the tangency points of its sides with a conic.}\label{Kasner}
\end{figure}

Moreover, the converse statement holds: being Poncelet is also necessary for a convex polygon~$P$ with an odd number of vertices to be projectively equivalent to~$T(P)$~\cite{Izosimov2019}.

Other interesting facts about behavior of inscribed polygons under the iterations of penta\-gram-like maps can be found in~\cite{schwartz2010elementary}.

And since we are talking about pentagons in the projective plane, let us conclude with a~theo\-rem by E.~Kasner illustrated in Fig.~\ref{Kasner}.

There are two operations on a pentagon: $T$, the pentagram map, and $I$, replacing the pentagon by the pentagon whose vertices are the tangency points of the conic that is tangent to its five sides. Kasner's theorem asserts that these two operations commute: $T\circ I = I \circ T$ (see~\cite{TaK} for details and a generalization to Poncelet polygons).

\subsection{Problem 19: Paper M\"obius bands}

Unlike the previous problems, this one is still open. The state of the art is as follows.

Let $\lambda$ be the number such that a smooth M\"obius band can be made out of $1\times l$ rectangle with $l>\lambda$ and cannot be made if $l < \lambda$. Then
\[
\frac{\pi}{2} \le \lambda \le \sqrt{3}.
\]
This theorem is contained in \cite{HW}; see \cite[Chapter~14]{FTb} for an accessible exposition. The reader who has ideas about an improvement of this result is encouraged to contact the authors.

This problem is closely related to curved origami, a ``hot'' topic of contemporary applied mathematical research.

%\cite{Ar-problems,schwartz1992pentagram,TaB}
\pdfbookmark[1]{References}{ref}
\LastPageEnding

\end{document}